\documentclass[12pt]{article}
\usepackage{amssymb}
\usepackage{amsfonts}
\usepackage{amsthm}
\usepackage{amsmath}
\usepackage{fullpage}
\usepackage{algorithm}
\usepackage{algorithmic}
\usepackage{mathtools}
\usepackage{mathrsfs}

\newtheorem{Lemma}{Lemma}

\newtheorem{Conjecture}{Conjecture}
\newtheorem{Theorem}{Theorem}

\newtheorem{Definition}{Definition}
\newtheorem{Corollary}{Corollary}

\newcommand{\vol}{\text{vol}}

\newcommand{\conv}{\text{conv}}

\usepackage[pdftex]{graphicx}
\graphicspath{{C:/Users/Sam/Desktop/PaperFigures/}}
\DeclareGraphicsExtensions{.pdf,.png}

\begin{document}
\title{Rectifications of Convex Polyhedra}
\author{Samuel Reid\thanks{Geometric Energy Corporation. $\mathsf{sam@geometricenergy.ca}$}\;$^{,}$\thanks{University of Calgary, Department of Mathematics \& Statistics.}}
\maketitle

\begin{abstract}
A convex polyhedron, that is, a compact subset of $\mathbb{R}^3$ which is the intersection of finitely 
many closed half-spaces, can be rectified by taking the convex hull of the midpoints of the edges 
of the polyhedron. We derive expressions for the side lengths and areas of rectifications of regular 
polygons in plane, and use these results to compute surface areas and volumes of various convex polyhedra.
We introduce rectification sequences, show that there are exactly two disjoint pure rectification sequences 
generated by the platonic solids, and formulate new results related to the Mahler conjecture.
\end{abstract}

\section{Rectifications of Convex Polygons}
We iterate the geometric process of rectification, or maximal truncation, on a seed polyhedron to obtain metric information regarding operations of polyhedra. We introduce the notion of rectification for polygons and polyhedra; rectifications of convex $d$-polytopes for dimensions $d \geq 4$ will be presented in future work. This topic leads off of the classical polyhedral geometry seen in the works of H.S.M. Coxeter \cite{CoxeterMiller},\cite{Coxeter}, Norman Johnson \cite{Johnson1},\cite{Johnson2}, and others such as Cromwell \cite{Cromwell}, Conway \cite{Conway}, and T\'{o}th \cite{Toth}. Intuitively known since antiquity, define:

\begin{Definition}
Let $P_{n} = \conv\{x_{1},...,x_{n}\} \subset \mathbb{R}^2$ be a convex polygon with $n$ vertices and $n$ edges. Then,
$$R_{1}[P_{n}] = \conv\left\{\frac{x_{i} + x_{i+1}}{2} \; | \; (x_{i},x_{i+1}) \in E(G_{P_{n}}), 1 \leq i \leq n, x_{n+1}=x_{1}.\right\},$$
where $E(G_{P_{n}})$ is the edge set of the graph $G_{P_{n}}$; the combinatorial structure of $P_{n}$.
\end{Definition}

\begin{figure}[h!]\label{polygonrectifications}
\begin{center}
\includegraphics[scale=0.25]{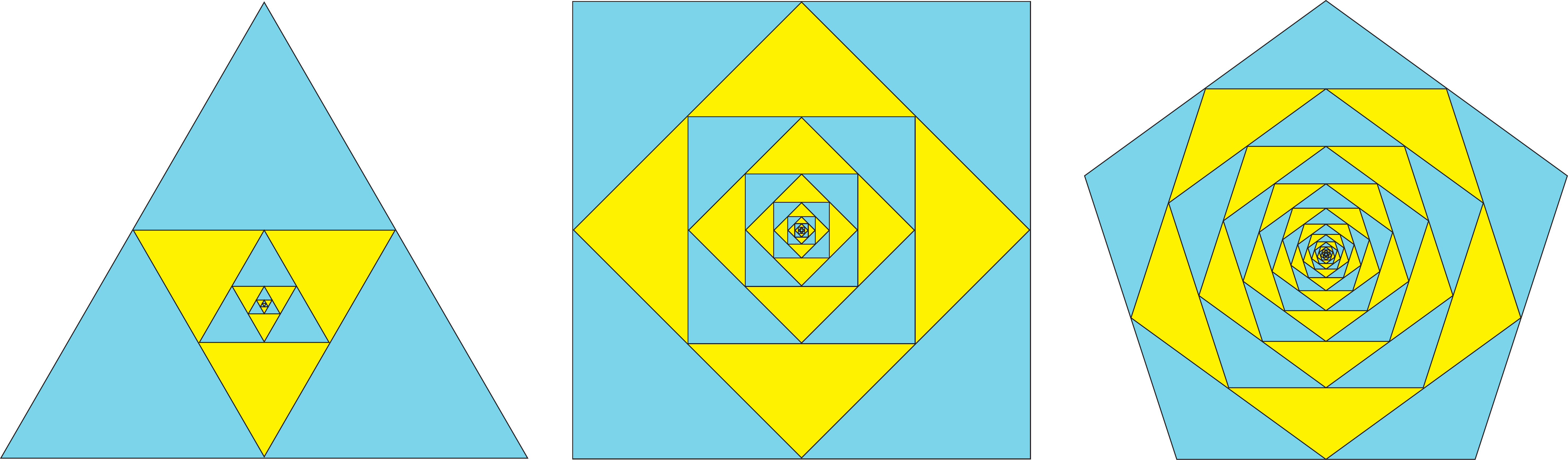}
\end{center}
\caption{The infinite family of polygon rectifications $R_{\infty}(P_{3})$, $R_{\infty}(P_{4})$, and $R_{\infty}(P_{5})$.}
\end{figure}

\begin{Theorem}\label{areathm}
Let $P_{n}$ denote a regular $n$-gon with $\text{area}(P_{n})=1$. Then,
$$\text{area}(R_{k}(P_{n})) = \left(\frac{1 - \cos(\theta_{n})}{2}\right)^{k}$$
and
$$\text{area}(R_{\infty}(P_{n})) = \frac{2}{1 + \cos(\theta_{n})}$$
where $\theta_{n} = \frac{\pi(n-2)}{n}$ is the interior angle of $P_{n}=R_{0}(P_{n})$, $R_{k}(P_{n})$ is the $k^{\text{th}}$ rectification of $P_{n}$ and
$$R_{\infty}(P_{n}) = \bigcup_{k=0}^{\infty} R_{k}(P_{n})$$
\end{Theorem}
\begin{proof}
Let $P_{n}$ be a regular $n$-gon with $\text{area}(P_{n})=1$. Then by $\text{area}(P_{n}) = \frac{1}{4} ns^2 \cot\left(\frac{\pi}{n}\right)$, where $s$ is the side length of $P_{n}$,
$$s^2 = \frac{4\tan\left(\frac{\pi}{n}\right)}{n}$$
By the law of cosines,
$$s_{1}^{2} = \left(\frac{s}{2}\right)^2 + \left(\frac{s}{2}\right)^2 - 2\left(\frac{s}{2}\right)^2 \cos(\theta_{n}) = \frac{s^2}{2}(1-\cos(\theta_{n})) = \frac{2\tan\left(\frac{\pi}{n}\right)}{n}(1-\cos(\theta_{n}))$$
where $\theta_{n} = \frac{\pi(n-2)}{n}$ is the interior angle of $P_{n}$ and $s_{1}$ is the side length of $R_{1}(P_{n})$. We can then compute the area of the first rectification of $P_{n}$ as follows,
$$\text{area}(R_{1}(P_{n})) = \frac{1}{4}ns_{1}^2 \cot\left(\frac{\pi}{n}\right) = \frac{1 - \cos(\theta_{n})}{2}$$
Now, let $k \geq 2$ be arbitrary and let $s_{k}$ be the side length of $R_{k}(P_{n})$. For induction, assume that
$$s_{k}^2 = \frac{\tan\left(\frac{\pi}{n}\right)}{2^{k-2}n}(1-\cos(\theta_{n}))^k$$
By the law of cosines,
$$s_{k+1}^{2} = \frac{s_{k}^2}{2}(1-\cos(\theta_{n})) = \frac{\tan\left(\frac{\pi}{n}\right)}{2^{k-1}n}(1-\cos(\theta_{n}))^{k+1}$$
Then,
$$\text{area}(R_{k}(P_{n})) = \frac{1}{4}ns_{k}^{2}\cot\left(\frac{\pi}{n}\right) = \left(\frac{1-\cos(\theta_{n})}{2}\right)^{k}$$
implies that
$$\text{area}(R_{k+1}(P_{n})) = \frac{1}{4}ns_{k+1}^{2}\cot\left(\frac{\pi}{n}\right) = \left(\frac{1 - \cos(\theta_{n})}{2}\right)^{k+1}$$
Therefore the formula for $\text{area}(R_{k}(P_{n}))$ holds for all $k \in \mathbb{N}$ by induction. We now compute the area of the infinite family of polygon rectifications $R_{\infty}(P_{n})$.
$$\text{area}(R_{\infty}(P_{n})) = \text{area}\left(\bigcup_{k=0}^{\infty} R_{k}(P_{n})\right) = \sum_{k=0}^{\infty} \text{area}(R_{k}(P_{n})) = \sum_{k=0}^{\infty} \left(\frac{1 - \cos(\theta_{n})}{2}\right)^{k} = \frac{2}{1 + \cos(\theta_{n})}$$
by the convergence of the geometric series since $\left|\frac{1 - \cos(\theta_{n})}{2}\right| < 1$ for all $P_{n}$ with $n \geq 3$.
\end{proof}

While the perimeter of the infinite rectification $R_{\infty}(P_{n})$ follows trivially from the proof of Theorem \ref{areathm}, it is in general a difficult problem to determine to determine the surface area of $R_{\infty}(\text{conv}\{P\})$ when $P \subset \mathbb{R}^d$ with $d \geq 3$.

\begin{Corollary}\label{perimcor}
Let $P_{n}$ be a regular $n$-gon with $\text{area}(P_{n})=1$. Then,
$$\text{Perimeter}(R_{\infty}(P_{n})) = \frac{2 \sqrt{n \tan(\pi /n)}}{1 - |\cos(\pi/n)|}$$
\end{Corollary}
\begin{proof}
We have by the proof of Theorem \ref{areathm} that $$s_{k} = \sqrt{\frac{\tan\left(\frac{\pi}{n}\right)}{2^{k-2}n}(1-\cos(\theta_{n}))^k}$$
thus, $$\text{Perimeter}(R_{\infty}(P_{n})) = \sum_{k=0}^{\infty} n s_{k} = 2\sqrt{n\tan\left(\frac{\pi}{n}\right)} \sum_{k=0}^{\infty} \left(\sqrt{\frac{1 - \cos(\theta_{n})}{2}}\right)^{k} = \frac{\sqrt{8n \tan\left(\frac{\pi}{n}\right)}}{\sqrt{2} - \sqrt{1 - \cos(\theta_{n})}}$$
by the convergence of the geometric series since $\left|\sqrt{\frac{1 - \cos(\theta_{n})}{2}}\right| < 1$ for all $P_{n}$ with $n \geq 3$. Since $n >0$ we have by the definition of $\theta_{n}$ that
$$\text{Perimeter}(R_{\infty}(P_{n})) = \frac{2 \sqrt{n \tan(\pi /n)}}{1 - |\cos(\pi/n)|}$$
\end{proof}

We thus have a complete classification of regular polygons rectification sequences and their area 
and perimeter. We now want to generalize this problem to classifying rectification sequences of 
polyhedra in $\mathbb{R}^{3}$ and computing their geometric measures.

\section{Rectifications of Convex Polyhedra}

\begin{Definition}
Let $P = \conv\{x_{1},...,x_{v}\} \subset \mathbb{R}^3$ be a convex polyhedron with $e$ edges and $f$ faces.
Then, $$R_{1}[P] = \conv\left\{\frac{x_{i} + x_{j}}{2} \; | \; (x_{i},x_{j}) \in E(G_P)\right\},$$
where $E(G_{P})$ is the edge set of the graph $G_{P}$; the combinatorial structure of $P$.
\end{Definition}

Rectifications can be thought of as the maximal truncation of a polyhedron which sends a vertex to a face. If we consider the truncation of each vertex of a polyhedron by a plane with a normal vector coinciding with the centroid ray of the polyhedron then we obtain the rectification of the polyhedron when each new edge created by the truncation touches another newly created edge.

\begin{figure}[h!]\label{rectifiedtet}
\begin{center}
\includegraphics[scale=1]{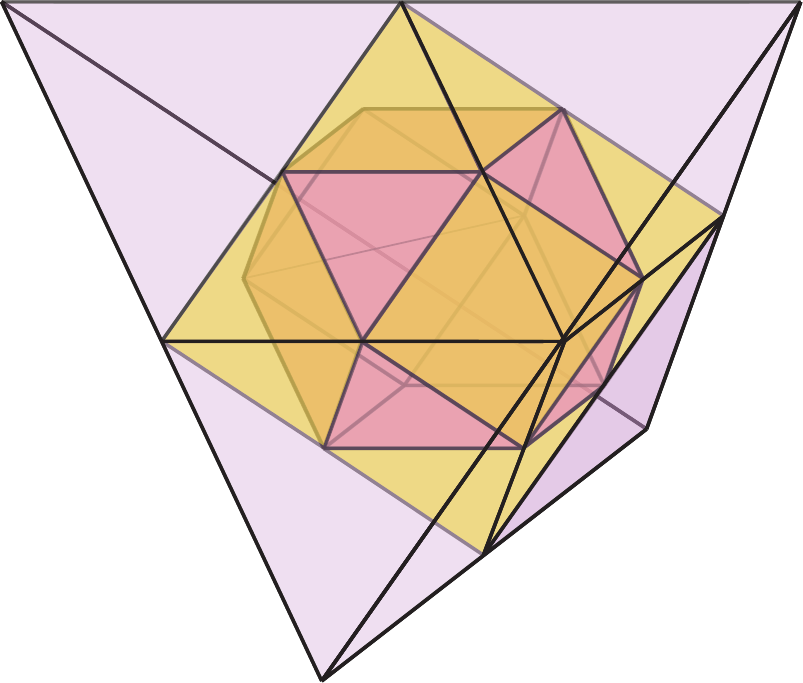}
\end{center}
\caption{The first two rectifications of a tetrahedron, namely an octahedron and a cuboctahedron inscribed inside eachother.}
\end{figure}

Rectifications of convex polyhedra can be easily understood combinatorially due to $V- E + F = \chi = 2$, as the rectification of a polyhedron can be thought of as a map between $f$-vectors
$$(v,e,f) \mapsto (e,2e,2+e)$$
There are two distinct infinite sequences of polytope rectifications which include the platonic solids; the first, denoted by $\eta_{s}$ includes the tetrahedron, octahedron, and cube and has an $f$-vector sequence given by
$$(4,6,4) \underbrace{\longrightarrow}_{\text{Tetrahedron}} (6,12,8) \underbrace{\longrightarrow}_{\text{Octahedron}} (12,24,14) \underbrace{\longrightarrow}_{\text{Cuboctahedron}} \cdot\cdot\cdot$$
and the second, denoted by $\xi_{s}$ includes the icosahedron and dodecahedron, which have the property that they have equal rectifications, and has an $f$-vector sequence given by
$$(12,30,20) \underbrace{\longrightarrow}_{\text{Icosahedron}} (30,60,32) \underbrace{\longrightarrow}_{\text{Icosidodahedron}} (60,120,62) \underbrace{\longrightarrow}_{\text{Icosidodecahedron}} \cdot\cdot\cdot$$

\begin{Theorem}
The polytope rectification sequences $\eta_{s}$ and $\xi_{s}$ are disjoint.
\end{Theorem}
\begin{proof}
By the combinatorial characterization of rectification as an $f$-vector map $(v,e,f) \mapsto (e,2e,2+e)$, an equivalent question is whether there exists a polytope in the sequence $\eta_{s}$ which also exists in $\xi_{s}$, which implies that the existence of a pair $(k,x)$ which satisfies
$$2^{k}12 = 2^{x}30$$
but this is never satisfied for $k,x\in\mathbb{Z}$ as $k-x =\frac{\log(5)}{\log(2)} - 1 \notin \mathbb{Q}$.
\end{proof}

\begin{figure}[h!]\label{prism1}
\begin{center}
\includegraphics[scale=0.5]{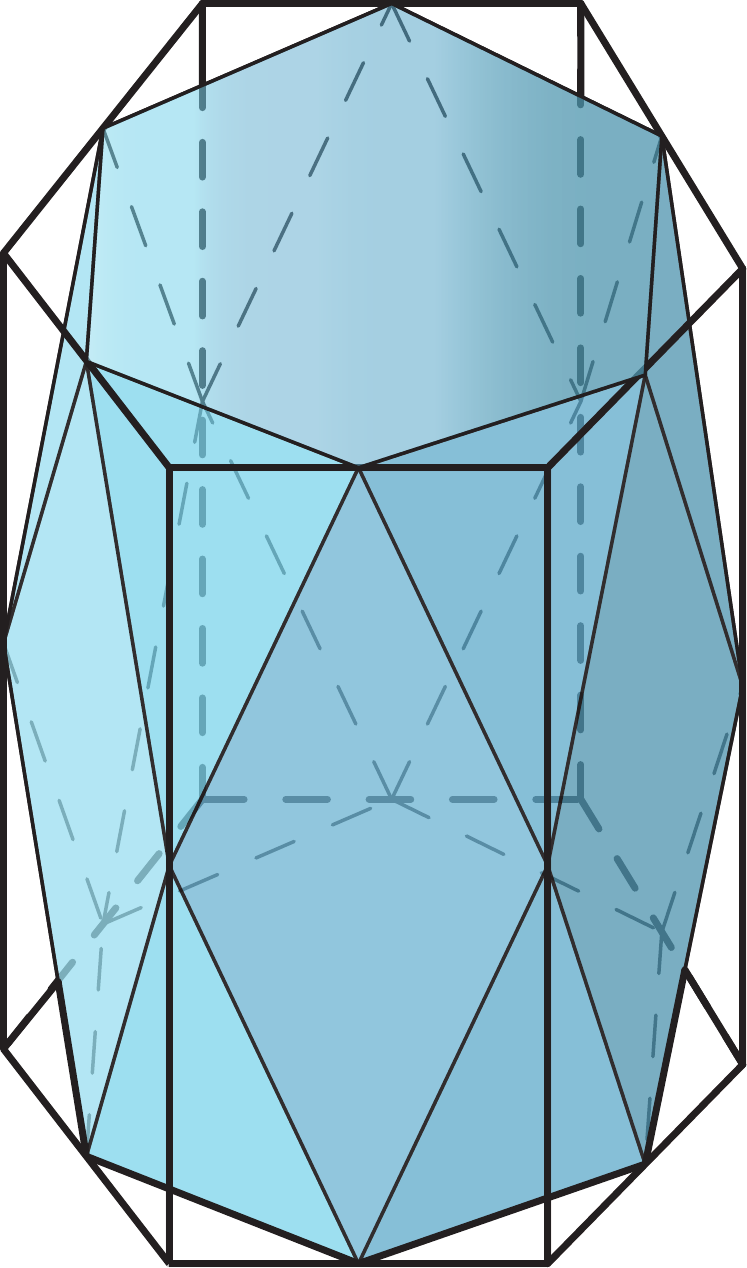}
\end{center}
\caption{The first rectification of the pentaprism.}
\end{figure}

We now calculate the volume of the first rectification of a prism generated from a regular polygon in the plane.
\begin{Theorem}
Let $P_{n}$ be a regular polygon of area 1 in the plane. Then,
$$\vol\left(R_{1}(P_{n} \times I)\right) = 1 - \frac{s^2}{48}\sin(\theta_{n})$$
where $s$ is the side length of $P_{n}$ and $\theta_{n} = \frac{\pi(n-2)}{n}$ is the internal angle of a regular $n$-gon.
\end{Theorem}

\section{Pure Rectification Sequences}
\begin{Definition}
A pure rectification sequence of length $k$ is a sequence
$$P \rightarrow R_{1}[P] \rightarrow \cdot\cdot\cdot \rightarrow R_{k-1}[P]$$
so that every element of the sequence is semiregular.
\end{Definition}

\begin{Theorem}
Let $$P \rightarrow R_{1}[P] \rightarrow \cdot\cdot\cdot \rightarrow R_{k-1}[P]$$ be a semiregular
rectification sequence. Then the f-vector and volume of $R_{k}[P]$ are
$$(f_{0}(R_{k}[P]),f_{1}(R_{k}[P]),f_{2}(R_{k}[P]))=(2^{k-1}f_{1}(P),2^{k}f_{1}(P),2+2^{k-1}f_{1}(P))$$
and
$$\vol\left(R_{k}[P]\right) = \vol(P) - \sum_{i=1}^{f_{0}(P)}\vol(\conv\{P\slash v_{i}^{0} \cup \{v_{i}^{0}\}\})
- \sum_{N=1}^{k} \sum_{j=1}^{2^{N-1}f_{1}(P)} \vol(\conv\{R_{N}[P]\slash v_{j}^{N} \cup \{v_{j}^{N}\}\})$$
where $P \slash v_{i}^{0}$ denotes the vertex figure of $P$ at $v_{i}^{0} \in V(P)$ truncated to half the edge length of $P$
and $R_{N}[P]\slash v_{j}^{N}$ denotes the vertex figure of $R_{N}[P]$ at the vertex $v_{j}^{N} \in V(R_{N}[P])$ truncated
to half the edge length of $R_{N}[P]$.
\end{Theorem}

\begin{Conjecture}
Let $P$ be a convex polyhedron. If $R_{1}[P]$ is contained in a pure rectification sequence then
$$R_{1}[P] = \lambda R_{1}[P^{\circ}], \exists \lambda >0.$$
\end{Conjecture}
\begin{Lemma}
Let $P$ be a convex polyhedron so that $R_{1}[P]=\lambda R_{1}[P^{\circ}], \exists \lambda >0$.
If $$\vol(R_{1}[P]) \geq \sqrt{\frac{32 \lambda}{3}}$$
then the Mahler conjecture holds for $R_{1}[P]$ and $R_{1}[P] \times R_{1}[P]^{\circ}$.
\end{Lemma}

\begin{Conjecture}[Mahler Conjecture for Pure Rectifications and a Class of Hanner 6-Polytopes]
If $R_{1}[P]$ is a pure rectification of a convex polyhedron $P$ then the Mahler conjecture is true
for $R_{1}[P]$ and $R_{1}[P] \times R_{1}[P]^{\circ}$.
\end{Conjecture}
\begin{proof}
Apply Conjecture 1 and Lemma 1.
\end{proof}

\begin{Conjecture}
Let $P \subset \mathbb{R}^3$ be a convex polyhedron. Then,
$$\exists \lambda \in \mathbb{R}^+ \; : \; \lambda P^{\circ} \cap P = R_{1}[P].$$
\end{Conjecture}

\end{document}